\pdfoutput=1\relax
\documentclass[reqno]{amsart}

\usepackage{verbatim}
\usepackage[textsize=scriptsize]{todonotes}
\usepackage{tikz-cd}
\usepackage{etoolbox}
\usepackage{etex}
\usepackage{hyperref}
\usepackage{cleveref}
\usepackage[T1]{fontenc}

\usepackage{chemarr}
\usepackage{amssymb}
\usepackage{amsmath}
\usepackage{comment}
\usepackage{spectralsequences}
\usepackage{cleveref}
\usepackage{mathtools}
\usepackage{rotating}
\usepackage{wrapfig}
\usepackage{outlines}
\usepackage{graphicx}
\usepackage{multicol}
\usepackage{bbm}

\theoremstyle{definition}
\newtheorem{nul}{}[section]
\newtheorem{dfn}[nul]{Definition}

\newtheorem{rmk}[nul]{Remark}

\newtheorem{cnv}[nul]{Convention}
\newtheorem{ntn}[nul]{Notation}
\newtheorem{exm}[nul]{Example}

\newtheorem{wrn}[nul]{Warning}
\newtheorem{qst}[nul]{Question}

\newtheorem{hyp}[nul]{Hypothesis}
\newtheorem*{dfn*}{Definition}
\newtheorem*{axm*}{Axiom}
\newtheorem*{ntn*}{Notation}
\newtheorem*{exm*}{Example}
\newtheorem*{exr*}{Exercise}
\newtheorem*{int*}{Intuition}
\newtheorem*{qst*}{Question}
\newtheorem*{rmk*}{Remark}

\theoremstyle{plain}

\newtheorem{thm}[nul]{Theorem}
\newtheorem{prop}[nul]{Proposition}
\newtheorem{lem}[nul]{Lemma}

\newtheorem{cnj}[nul]{Conjecture}

\newtheorem*{thm*}{Theorem}
\newtheorem*{prop*}{Proposition}
\newtheorem*{cor*}{Corollary}
\newtheorem*{lem*}{Lemma}
\newtheorem*{cnj*}{Conjecture}

\let\oldwidetilde\widetilde
\protected\def\widetilde{\oldwidetilde}

\DeclareMathOperator*{\colim}{\mathrm{colim}}

\DeclareMathOperator{\Ss}{\mathbb{S}}

\DeclareMathOperator{\F}{\mathbb{F}}

\DeclareMathOperator{\A}{\mathcal{A}}
\DeclareMathOperator{\C}{\mathcal{C}}
\DeclareMathOperator{\D}{\mathcal{D}}

\DeclareMathOperator{\MO}{\mathrm{MO}}

\DeclareMathOperator{\MOn}{\mathrm{MO \langle n \rangle}}

\DeclareMathOperator{\bS}{\mathbb{S}}

\DeclareMathOperator{\Ext}{\mathrm{Ext}}

\DeclareMathOperator{\Sp}{\mathrm{Sp}}

\newcommand{\spO}{\Sigma^{\infty}_+ \mathrm{O \langle n-1\rangle}}

\newcommand{\wt}{\widetilde}

\newcommand{\BO}{\mathrm{BO}}

\newcommand{\BP}{\mathrm{BP}}
\newcommand{\Syn}{\mathrm{Syn}}

\newcommand{\MSpin}{\mathrm{MSpin}}

\newcommand{\Z}{\mathbb{Z}}

\def\on{\mathrm{o}\langle n-1\rangle}
\def\sOf{\Sigma^{\infty} \mathrm{O} \langle n-1\rangle}

\DeclarePairedDelimiter\abs{\lvert}{\rvert}%
\makeatletter
\let\oldabs\abs
\def\abs{\@ifstar{\oldabs}{\oldabs*}}

\usepackage{tikz}
\usetikzlibrary{matrix,arrows,decorations}
\usepackage{tikz-cd}

\usepackage{adjustbox}

\let\oldtocsection=\tocsection
 
\let\oldtocsubsection=\tocsubsection
 
\let\oldtocsubsubsection=\tocsubsubsection
 
\renewcommand{\tocsection}[2]{\hspace{0em}\oldtocsection{#1}{#2}}
\renewcommand{\tocsubsection}[2]{\hspace{1em}\oldtocsubsection{#1}{#2}}
\renewcommand{\tocsubsubsection}[2]{\hspace{2em}\oldtocsubsubsection{#1}{#2}}

\usepackage{etoolbox}
\newtoggle{draft}
\togglefalse{draft}

\iftoggle{draft} {
\usepackage[margin=1.5in]{geometry}
\newcommand{\NB}[1]{\todo[color=gray!40]{#1}}
\newcommand{\TODO}[1]{\todo[color=red]{#1}}
}{ 
\usepackage[margin=1.5in]{geometry}
\newcommand{\NB}[1]{}
\newcommand{\TODO}[1]{}
\renewcommand{\todo}[1]{}
\renewcommand{\todo}[1]{}
}
\title{Inertia groups in the metastable range}

\author{Robert Burklund}
\address{Department of Mathematics, MIT, Cambridge, MA, USA}
\email{burklund@mit.edu}

\author{Jeremy Hahn}
\address{Department of Mathematics, MIT, Cambridge, MA, USA}
\email{jhahn01@mit.edu}

\author{Andrew Senger}
\address{Department of Mathematics, Harvard University, Cambridge, MA, USA}
\email{senger@math.harvard.edu}

\begin{document}

\begin{abstract}
We prove that the inertia groups of all sufficiently-connected, high-dimensional $(2n)$-manifolds are trivial. This is a key step toward a general classification of manifolds in the metastable range.  Specifically, for $m \gg 0$ and $k>5/12$, suppose $M$ is a $\lfloor km \rfloor$-connected, smooth, closed, oriented $m$-manifold and $\Sigma$ is an exotic $m$-sphere.  We prove that, if $M \sharp \Sigma$ is diffeomorphic to $M$, then $\Sigma$ bounds a parallelizable manifold. Our proof is built on an understanding of the second extended power functor in Pstr\k{a}gowski's category of synthetic spectra.
\end{abstract}
\maketitle

\setcounter{tocdepth}{1}
\tableofcontents
\vbadness 5000


\section{Introduction}\label{sec:intro}
An $m$-dimensional \emph{exotic sphere} $\Sigma$ is a smooth, oriented manifold that is homeomorphic, but not necessarily diffeomorphic, to $S^m$.
When $m \ge 5$, exotic spheres up to orientation-preserving diffeomorphism constitute the Kervaire--Milnor group $\Theta_m$, with group operation given by connected sum.
Kervaire and Milnor proved that there is an exact sequence
\[0 \to \mathrm{bP}_{m+1} \to \Theta_m \to \mathrm{coker}(J)_m,\]
where $\mathrm{bP}_{m+1}$ is the subgroup of exotic spheres that bound parallelizable $(m+1)$-manifolds.
When $m$ is even, $\mathrm{bP}_{m+1}$ contains only the standard sphere $S^m$.

If $M$ is any smooth, oriented $m$-manifold, and $\Sigma \in \Theta_m$, then we may form the connected sum $M \sharp \Sigma$.
Often, this construction will change the diffeomorphism type of $M$, but this need not always be the case.

\begin{dfn}
Suppose $M$ is a smooth, closed, oriented $m$-manifold.
Then the \emph{inertia group} of $M$, denoted $I(M)$, is the subgroup of $\Theta_m$ consisting of all exotic spheres $\Sigma$ such that $M \sharp \Sigma$ is diffeomorphic to $M$ via an orientation-preserving diffeomorphism.
\end{dfn}


Essentially by definition, the inertia group of any exotic sphere is trivial.
On the other hand, Winkelnkemper constructed manifolds of every dimension with $I(M)=\Theta_m$ \cite{Winkelnkemper}. In general, both precise calculations and bounds on the size of $I(M)$ are known only by hard work in special circumstances. 
There have been explicit calculations of inertia groups of low dimensional complex and quaternionic projective spaces \cite{CPinertiaI, QuatInertia, CPinertiaII}, as well as of certain products of spheres with one another and with low dimensional complex projective spaces \cite{SchultzProduct, Souls}.
The inertia groups of hyperbolic manifolds are the subject of striking results of Farrell and Jones \cite{FarrellJones,FarrellJonesII}, and a sample of additional results may be found in \cite{Naoum, Frame}.

The most general class of manifolds for which inertia groups are reasonably understood are the \emph{highly connected manifolds}.
A $(2n)$-dimensional manifold is said to be highly connected if it is $(n-1)$-connected, meaning it is simply connected and its first $(n-1)$ integral homology groups vanish.  
In the 1960s, Wall \cite{WallInertia} \cite[\S 16]{Wall67} and Kosinski \cite{Kosinski} proved the following:

\begin{thm}[Wall, Kosinski]
Suppose that $M$ is a stably frameable $(n-1)$-connected, smooth, closed $(2n)$-manifold, with $n \ge 3$. Then $I(M)=0$.
\end{thm}

The assumption that $M$ be stably frameable is essential in the above result, as the following examples show:

\begin{thm}[Kramer--Stolz \cite{KramerStolz}]
The inertia group of $\mathbb{HP}^{2}$ is all of $\Theta_8$, and the inertia group of $\mathbb{OP}^{2}$ is all of $\Theta_{16}$.  Neither of these groups are trivial.
\end{thm}

The projective planes $\mathbb{HP}^{2}$ and $\mathbb{OP}^{2}$ are far from generic examples of highly connected manifolds.
Indeed, the solution \cite{AdamsJIV} of the Hopf Invariant $1$ question precludes any similar objects in larger dimensions.
In \cite{Boundaries}, the authors proved that this is no accident.  As part of that work, we showed that any $(n-1)$-connected $(2n)$-manifold of dimension larger than $464$ has trivial inertia group, whether or not it is stably frameable.
This result is an enhancement of
previous work of Stephan Stolz \cite[Theorem D]{StolzBook}, which relies on unpublished work of Mahowald \cite{MahTokyo} (see \cite{Changv1}) and resolved the cases where $n \not \equiv 1$ modulo $8$.

With hard work, the triviality of inertia groups has also been proven for restricted classes of $(n-1)$-connected $(2n+1)$,$(2n+2)$, $(2n+3)$, and $(2n+4)$-manifolds \cite{Kosinski} \cite{Wilkens} \cite[Theorems B,D]{StolzBook}.

In this paper, we give the first general result on the inertia groups of $\lfloor km \rfloor$-connected $m$-manifolds for any $k<1/2$. Specifically, we study $\lfloor km \rfloor$-connected $m$-manifolds for real numbers $k>5/12$. 
We find, for each such real number $k$, an integer $N_k$ such that the following theorem is true:


\begin{thm} \label{thm:intro-inertia}
Suppose $M$ is a $\lfloor km \rfloor$-connected, smooth, closed, oriented manifold of dimension $m \ge N_k$.  Then any exotic sphere $\Sigma \in I(M)$ must bound a parallelizable manifold.  In particular, if $M$ is even-dimensional then $I(M)=0$, and if $M$ is odd-dimensional then $I(M)$ is a subgroup of the cyclic group $bP_{m+1}$. \footnote{The order of $\mathrm{bP}_{m+1}$ is known except when $m=125$, where the last remaining case of the Kervaire Invariant one problem remains unresolved.}
\end{thm}


In order to prove \Cref{thm:intro-inertia}, we must push the techniques introduced in \cite{Boundaries} solidly into the metastable range, and we require deeper homotopy theoretic techniques to do so.  In \cite{Boundaries}, we used the functor $\nu:\Sp \to \Syn_{\mathbb{F}_p}$ to study the interaction of Adams filtration with commutativity up to coherent homotopy.  Specifically, we used the fact that, for any class $x$ in the homotopy of any $\mathbb{E}_{\infty}$ ring $R$, there is a canonical homotopy witnessing the relation $x^2=(-1)^{|x|} x^2$.  This canonical homotopy is only the first of the infinity of coherent homotopies that constitutes an $\mathbb{E}_{\infty}$ ring structure.  In this paper, we must interact with the $\mathbb{E}_{\infty}$ ring structure at a deeper level; for example, the difference between $\nu$ of a free $\mathbb{E}_{\infty}$ ring and the free $\mathbb{E}_{\infty}$ algebra on $\nu$ of a spectrum becomes very relevant.

\begin{rmk}
An explicit value for $N_k$ can be extracted from our proof whenever $k>13/30$.
For example, \Cref{thm:mon-unit} shows that one may take $N_{0.45}$ to be $750$, or even a slightly smaller integer. \todo{this is an overly conservative estimate}

For $13/30 \ge k>5/12$, we express $N_k$ in terms of the unknown (and therefore inexplicit) vanishing curve on the $\mathrm{E}_{\infty}$-page of the Adams-Novikov spectral sequence for the homotopy groups of spheres.
That such a vanishing curve exists, and is sublinear, is the main technical result powering the Nilpotence Theorem of Devinatz--Hopkins--Smith \cite{DHS}.
\end{rmk}

As we recall in \Cref{sec:inertia}, a classic geometric argument allows us to deduce \Cref{thm:intro-inertia} as a consequence of the following theorem:

\begin{thm} \label{thm:intro-boundary}
Suppose $M$ is a $\lfloor km \rfloor$-connected smooth $(m+1)$-manifold with boundary homeomorphic to a sphere.  Then, if $m \ge N_k$, $\partial M$ also bounds a parallelizable manifold.
\end{thm}

\Cref{thm:intro-boundary} has the additional geometric consequence that, for $m \gg 0$, any $\lfloor km \rfloor$-connected topological $(2m+1)$-manifold that is smoothable away from a disk must be smoothable.

As we explain in \Cref{sec:inertia}, the Pontryagin--Thom construction allows us to deduce \Cref{thm:intro-boundary} from the following purely homotopy-theoretic result:

\begin{thm} \label{thm:intro-MO}
  There exists a sublinear function $\epsilon(n)$ such that, if
  $0 \leq d \leq \frac{2}{5}n - \epsilon(n)$,
  then the kernel of the unit map
  \[ \pi_{2n+d}\Ss \to \pi_{2n+d}\MOn \]
  is equal to the degree $2n+d$ part of the image of $J$.
\end{thm}

A more precise version of \Cref{thm:intro-MO} is stated as \Cref{thm:mon-unit}, and its proof occupies all of the paper subsequent to \Cref{sec:inertia}.

\begin{rmk}
Our results were anticipated in a 1985 book by Stephan Stolz \cite{StolzBook}, who sketched a proof that \Cref{thm:intro-MO} would follow from certain classic conjectures about the mod $2$ Adams spectral sequence for the $2$-completed sphere \cite[pp. XX-XXI]{StolzBook}.

\end{rmk}

We are hopeful that one day Stolz's original sketch may be realized, and specifically that the following will be proved:

\begin{cnj} \label{cnj:intro-Adams}
The $\mathrm{E}_{\infty}$-page of the $\mathbb{F}_2$-Adams spectral sequence for the sphere admits a line of slope $1/6$ above which every class is $v_1$-periodic.
\end{cnj}

The reader may note that the existence of $v_2$-periodic families precludes any such line of slope less than $1/6$.
The direct odd primary analog of \Cref{cnj:intro-Adams} is a result of the first author \cite[Proposition 6.3.20]{cookware}, but the prime $2$ is substantially more subtle.

Assuming \Cref{cnj:intro-Adams}, one follows \cite[loc. cit.]{StolzBook} to obtain the following improvement on \Cref{thm:intro-MO}:

\begin{cnj} \label{cnj:optimal}
  There exists a sublinear function $\epsilon(n)$ such that, if
  $0 \leq d \leq n - \epsilon(n)$,
  then the kernel of the unit map
  \[ \pi_{2n+d}\Ss \to \pi_{2n+d}\MOn \]
  is equal to the degree $2n+d$ part of the image of $J$.
\end{cnj}

\begin{rmk}
The distinction between the unproven \Cref{cnj:optimal} and our main \Cref{thm:intro-MO} is the improvement of a number from $2/5$ to $1$.
We believe that the analogous statement is false for any real number larger than $1$, and this seems to be an interesting open problem.
For comments, see \Cref{qst:asymp}.
\end{rmk}

\begin{rmk}
Stolz's arguments, given \Cref{cnj:intro-Adams}, are substantially more elementary than the ones in this paper, avoiding use of both higher algebra and synthetic spectra.  However, even assuming \Cref{cnj:intro-Adams}, more sophisticated arguments seem necessary to obtain optimal values for the integers $N_k$.  In particular, the bounds we give on $N_k$ in our more precise \Cref{thm:mon-unit} are substantially stronger than the bounds that can be obtained by Stolz's argument.

It is particularly desirable to obtain sharp bounds when attempting to completely classify highly connected manifolds. This is because, if $n \ge 3$ and $n \ne 63$,\footnote{In dimension $126$, the Kervaire invariant one problem remains open.} then the only remaining obstacle to a full classification of $(n-1)$-connected $(2n)$-manifolds is the determination of certain inertia groups \cite{HighDimGeo}.  In fact, these inertia groups will be computed in all dimensions in forthcoming work of the last author and Adela YiYu Zhang, using a careful combination of the techniques of this paper with those of \cite{HighDimGeo}.
A special case of Crowley's $Q$-form conjecture, proved in the PhD thesis of Nagy \cite{NagyThesis}, provides the geometric input necessary to reduce the study of inertia groups of highly connected manifolds entirely to problems of homotopy theory.
Indeed, Crowley and Nagy have announced a complete determination of the inertia groups of $3$-connected $8$-manifolds \cite{CrowleyNagy}, and Crowley, Olbermann, and Teichner have independent work in progress determining the inertia groups of all $7$-connected $16$-manifolds.
\end{rmk}



\subsection*{Acknowledgements:} We thank Diarmuid Crowley, Manuel Krannich, Alexander Kupers, Haynes Miller, and Oscar Randal-Williams for helpful conversations. We would also like to thank Adela YiYu Zhang for useful feedback on a previous version of this work. The middle author is supported by NSF DMS-1803273, and the last author by an NSF GRFP fellowship under Grant No. 1745302.

\subsection*{Conventions:} To keep this paper reasonably succinct, we assume the reader to be comfortable with both higher algebra in the language of \cite{HA} and with the $\mathbb{F}_p$-synthetic spectra of \cite{Pstragowski}.  For an introduction to synthetic spectra using our notational conventions, the reader is referred to \cite[\S 9]{Boundaries}.  We shall also have occasion to reference results from Sections $5$ and $10$, as well as Appendix B, of \cite{Boundaries}.


\section{The geometry of inertia groups}\label{sec:inertia}
For any integer $n \ge 3$, we define $\MOn$ to be the Thom spectrum associated to the canonical map $\tau_{\ge n} \BO \to \BO$.
In \Cref{sec:mon-unit} - \Cref{sec:lines}, we will find conditions on integers $n \ge 3$ and $d \ge 0$ such that the following hypothesis holds:

\begin{hyp} \label{hyp}
The kernel of the unit map
\[\pi_{2n+d} \Ss \to \pi_{2n+d} \MOn\]
is exactly the degree $2n+d$ component of the image of $J$.
\end{hyp}

In this section, we explain, for fixed $n$ and $d$, the geometric consequences of \Cref{hyp}.
In terms of the results stated in the Introduction, we prove that \Cref{thm:intro-MO} implies \Cref{thm:intro-boundary}, which in turn implies \Cref{thm:intro-inertia}.

First, let us note the following straightforward result, first spelled out explicitly by Stolz.
\begin{prop}[Stolz]
Suppose \Cref{hyp}, and let $M$ be an $(n-1)$-connected, smooth, compact, oriented $(2n+d+1)$-manifold with boundary an exotic sphere $\Sigma$.  Then $\Sigma$ must bound a parallelizable manifold.
\end{prop}

\begin{proof}
This is \cite[Satz 1.7]{StolzBook}.  The basic idea is to consider the long exact sequence associated to the unit map of $\MOn$, which is of the form
\[\pi_{2n+d+1} \Ss \to \pi_{2n+d+1}\MOn \to \pi_{2n+d+1} \left(\MOn/\Ss \right) \to \pi_{2n+d} \Ss.\]
By the Pontryagin-Thom construction, classes in $\pi_{2n+d+1} \left(\MOn/\Ss \right)$ may be interpreted as bordism classes of $(n-1)$-connected, smooth $(2n+d)$-manifolds that have boundary an exotic sphere equipped with some framing.  The map to $\pi_{2n+d} \Ss$ then records the class of that exotic sphere in framed bordism.

We may arbitrarily equip the boundary $\Sigma$ of the manifold $M$ in question with a framing, thus obtaining a class in $\pi_{2n+d+1} \left(\MOn/\Ss \right)$.
\Cref{hyp} ensures that the image in $\pi_{2n+d} \Ss$ must be in the image of $J$, which means it is framed cobordant to a standard sphere equipped with some framing.
It therefore follows from Kervaire--Milnor \cite{KervaireMilnor} that $\Sigma$ bounds a parallelizable manifold.
\end{proof}

Next, we recall a classic geometric argument using the construction of a \emph{modified mapping torus}.  To the best of our knowledge, an argument of this form first appeared in \cite{WallInertia} (cf. \cite[15.5--15.10]{StolzBook}).

\begin{prop}
Suppose that $M$ is an $(n-1)$-connected, smooth, closed, oriented $(2n+d)$-manifold, and that $\Sigma \in \Theta_{2n+d}$ is a class in the inertia group of $M$.  Then $\Sigma$ is the boundary of an $(n-1)$-connected, smooth, oriented $(2n+d+1)$-manifold.
\end{prop}

\begin{proof}
Consider $\Sigma$ as obtained by gluing two $(2n+d)$-dimensional disks along an orientation-preserving diffeomorphism $f$ of $S^{2n+d-1}$.  We may assume that $f$ is the identity when restricted to the upper hemisphere $D^{2n+d-1}_+ \subset S^{2n+d-1}$.

Let $N$ denote $M$ with a disk removed---specifically, we remove the disk used to form the connected sum $M \sharp \Sigma$.
Then $\partial N$ is diffeomorphic to $S^{2n+d-1}$, and the assumption that $\Sigma$ is in the inertia group of $M$ may be rephrased as the existence of an orientation-preserving diffeomorphism $g:N \to N$ such that $g|_{\partial N} = f$.

Now, we may form a new manifold $T$ from $N \times I$ by identifying $(n,1)$ with $(g(n),0)$.  By our assumption on $f$, $D^{2n+d-1}_+ \times S^1 \subset \partial T$.  We may thus attach $D^{2n+d-1} \times D^2$ along that part of the boundary, obtaining an $(n-1)$-connected $(2n+d+1)$-manifold with boundary diffeomorphic to $\Sigma$.

For more details, see either \cite[pp. 1-2]{WallInertia} or \cite[\S 15]{StolzBook}
\end{proof}

The results of this section motivate the following definition and question:
\begin{dfn}
For each integer $n \ge 3$, let $\alpha(n)$ denote the smallest integer $k$ such that the kernel
\[\mathrm{ker}\left(\pi_{k} \Ss \to \pi_{k} \MOn\right)\]
contains a class not in the image of $J$.

For a given prime $p$, we use $\alpha_p(n)$ to denote the least $k$ such that 
\[\mathrm{ker}\left(\pi_{k} \Ss_{(p)} \to \pi_{k} \MOn_{(p)}\right)\]
contains a class not in the $p$-localized image of $J$.
\end{dfn}

\begin{qst} \label{qst:asymp}
What is the asymptotic behavior of $\alpha(n)$?  Is there an infinite increasing sequence of dimensions, $n_1 \le n_2 \le \cdots$, such that 
\[\mathrm{lim}_{k \to \infty} \frac{\alpha(n_k)}{n_k} = 3?\]
\Cref{cnj:optimal} predicts that $\mathrm{liminf}_{n \to \infty} \frac{\alpha(n)}{n} \ge 3,$ but there does not appear to be any known estimate of the $\mathrm{limsup}$.
It may be enlightening, and still non-trivial, to determine the asymptotic behavior of $\alpha_p(n)$ for an odd prime $p$.
It seems likely that the asymptotics of $\alpha_p(n)$ are controlled by the $v_1$-banded vanishing line  in the mod $p$ Adams spectral sequence for the $p$-completed sphere, which is of slope $\frac{1}{|v_2|}$ \cite[Proposition 6.3.20]{cookware}.  Specifically, we conjecture that the $\mathrm{liminf}$ of $\frac{\alpha_{p}(n)}{n}$ is equal to $\frac{|v_2|}{|v_1|}$.
\end{qst}

\begin{rmk}
The arguments of this paper, and those of \cite[p. XX]{StolzBook}, suggest that any element in the kernel of the unit map for $\MOn$ 
which is not in the image of $J$ must be of relatively large Adams filtration.
However, it is not the case that a generic element of large Adams filtration must die in the homotopy of $\MOn$.

Indeed, many well-understood elements of large Adams filtration are detected by the topological modular forms spectrum $\mathrm{tmf}.$
The Ando--Hopkins--Rezk string orientation $\mathrm{MO}\langle 8 \rangle \to \mathrm{tmf}$ ensures that classes detected in $\mathrm{tmf}$ must also be detected by $\MOn$ whenever $n \ge 8$ \cite{Ando-Hopkins-Rezk}.
\end{rmk}

\section{The bar spectral sequence approach to the unit of $\MOn$}\label{sec:mon-unit}
The remaining sections are devoted to a homotopy theoretic proof of the following theorem, which is our main result.

\begin{thm}\label{thm:mon-unit-simple}
  There exists a sublinear function $\epsilon(n)$ such that if
  $-1 \leq d \leq \frac{2}{5}n - \epsilon(n)$,
  then the kernel of the unit map
  \[ \pi_{2n+d}\Ss \to \pi_{2n+d}\MOn \]
  is equal to the degree $2n+d$ part of the image of $J$.
\end{thm}

Our argument is a direct generalization of techniques in \cite{Boundaries}.  A quick outline follows:
\begin{enumerate}
\item At the end of this \Cref{sec:mon-unit}, we recall how \cite[Theorem 5.2]{Boundaries} reduces \Cref{thm:mon-unit-simple} 
   to a question about the suspension spectrum $\Sigma^{\infty} \mathrm{O} \langle n -1 \rangle$.
\item In \Cref{sec:background} we recall a bit of additional background.
Specifically, we first recall how Goodwillie calculus provides a canonical filtration of $\Sigma^{\infty} \mathrm{O} \langle n-1 \rangle$.
We then recall \cite[Lemma 10.18 and Lemma 10.19]{Boundaries}, which concern the $\mathbb{F}_p$-Adams filtrations of a map of spectra $J:\Sigma^{\infty} \mathrm{O} \langle n-1 \rangle \to \Ss$.
\item \Cref{sec:homotopy-argument} forms the technical heart of the paper.  We lift a diagram of spectra to a diagram of synthetic spectra, generalizing the main ideas of \cite[\S10]{Boundaries}.  The results of our arguments are lower bounds on the $\F_p$-Adams filtrations of elements in the kernel of $\pi_{2n+d}\Ss \to \pi_{2n+d}\MOn$.
\item In the final \Cref{sec:lines}, we apply known vanishing lines in the $\F_p$-Adams spectral sequences for spheres to complete the proof.
\end{enumerate}

In addition to our brief recollections of the main results of \cite[\S 5,\S 10]{Boundaries}, we will assume the reader has a general familiarity with both Goodwillie calculus and synthetic spectra. For background we suggest the reader consult \cite[Chapter $6$]{HA}, \cite{KuhnTAQ}, and \cite[\S9]{Boundaries}. 


We will prove \Cref{thm:mon-unit-simple} as a corollary of the following more comprehensive statement, which treats the $p$-local components of the problem separately:

\begin{thm}\label{thm:mon-unit}
  Let $K$ denote the kernel of the unit map
  \[ \pi_{2n+d}\Ss \to \pi_{2n+d}\MOn, \]
  and $K_{(p)}$ its $p$-local summand.
  These groups satisfy the following conditions:
  \begin{enumerate}
  \item If $-n-1 \le d$, then $\mathrm{Im}(J) \subseteq K$.\footnote{See, e.g., \cite[Lemma 6.1(ii)]{GRAbelianQuotients}.}
  \item There exists a sublinear function $\epsilon(n)$ such that if $2 \leq d \leq \frac{2}{5}n - \epsilon(n)$, then
    \[ K_{(2)} = \mathrm{Im}(J)_{(2)}. \]
  \item If $p = 2$, $n \geq 6$ and $2 \le d < \frac{4}{13}n - \frac{60}{13} - \frac{30}{13} \log_2(3n-2)$,
    then $K_{(2)} = \mathrm{Im}(J)_{(2)}$.
  \item If $p = 3$, $n \geq 60$ and $2 \leq d \leq n-4$, then $K_{(3)} = \mathrm{Im}(J)_{(3)}$.
  \item If $p = 5$, $n \geq 176$ and $2 \leq d \leq n-4$, then $K_{(5)} = \mathrm{Im}(J)_{(5)}$.
  \item If $p = 7$, $n \geq 120$ and $2 \leq d \leq n-4$, then $K_{(7)} = \mathrm{Im}(J)_{(7)}$.
  \item If $p = 11$, $n \geq 100$ and $2 \leq d \leq n-4$, then $K_{(11)} = \mathrm{Im}(J)_{(11)}$.
  \item If $p = 13$, $n \geq 120$ and $2 \leq d \leq n-4$, then $K_{(13)} = \mathrm{Im}(J)_{(13)}$.    
  \item If $p \geq 17$ and $2 \leq d \leq n-4$, then $K_{(p)} = \mathrm{Im}(J)_{(p)}$.
  \end{enumerate}
\end{thm}

\begin{rmk}
\Cref{thm:mon-unit} does not treat the cases $d=-1$, $d=0$ and $d=1$ of \Cref{thm:mon-unit-simple}.  However, these cases were analyzed in detail in \cite[Theorem 1.1]{Boundaries}.  In particular, $\mathrm{ker}\left(\pi_{2n-1}\Ss \to \pi_{2n-1} \MOn\right)$, $\mathrm{ker}\left(\pi_{2n}\Ss \to \pi_{2n} \MOn\right)$ and $\mathrm{ker}\left(\pi_{2n+1}\Ss \to \pi_{2n+1} \MOn\right)$ are known to agree with $\mathrm{Im}(J)$ when $n > 232$. 

The main theorem of \cite{HighDimGeo} is a complete analysis of $\mathrm{ker}\left(\pi_{2n-1}\Ss \to \pi_{2n-1} \MOn\right)$ for all values of $n$.  When $d<-1$, the analysis of $\mathrm{ker}\left(\pi_{2n+d}\Ss \to \pi_{2n+d} \MOn\right)$ is comparatively straightforward.
\end{rmk}

\begin{rmk}
\Cref{thm:mon-unit} implies that, if $2 \leq d \leq n-2$ and $n \ge 176$, then the cokernel of $\mathrm{Im}(J)$ in $\mathrm{ker}\left(\pi_{2n+d}\Ss \to \pi_{2n+d} \MOn\right)$ must be $2$-local.
\end{rmk}

\begin{ntn}
  Throughout the remainder of the paper we establish the following notation, where $n \ge 3$ is an integer:
  \begin{itemize}
  \item $\on \coloneqq \tau_{\geq n-1} \Sigma^{-1} \mathrm{bo}$,
	\item $\mathrm{O} \langle n-1 \rangle \coloneqq \Omega^{\infty} \on$,
  \end{itemize}  
\end{ntn}

\subsection{Manipulating the bar spectral sequence}

Recall that $\MOn$ is by definition the Thom spectrum of the map $\tau_{\ge n} \mathrm{BO} \to \mathrm{BO}$.  Looping this map once, one obtains a map $\mathrm{O} \langle n-1 \rangle \to \mathrm{O}$.  We may compose with the classical $J$ homomorphism $\mathrm{O} \to \Omega^{\infty} \Ss$ and apply the $(\Sigma^{\infty},\Omega^{\infty})$ adjunction to obtain a map $\Sigma^{\infty} \mathrm{O} \langle n-1\rangle \to \Ss$.

\begin{ntn}
We denote by $J$ the above map $\sOf \to \Ss$.
\end{ntn}

The map $J$ is naturally one of non-unital $\mathbb{E}_\infty$-ring spectra, and there is an associated unital $\mathbb{E}_\infty$-ring map $J_+:\Sigma^{\infty}_+ \mathrm{O}\langle n-1 \rangle \to \Ss$.

According to \cite{ABGHR}, the Thom spectrum $\MOn$ can be presented as a relative tensor product $\Ss \otimes_{\spO} \Ss$, where the action on the left is given by the augmentation and the action on the right by $J_+$.
In \cite{Boundaries}, the authors manipulated this tensor product expression to prove the following:

\begin{prop}[{\cite[Theorem 5.2]{Boundaries}}]
  After applying $\tau_{\leq 3n - 2}$, there is an equivalence of $\mathbb{E}_0$-algebras between $\MOn$ and the three term complex produced from
  \begin{center}\begin{tikzcd}[sep=huge]
      \sOf^{\otimes 2} \ar[r, "m - 1 \otimes J"] \ar[rr, bend left, "0"{name=D}] & \sOf \ar[r, "J"] \ar[Leftrightarrow, from=D, "\mathrm{can}"] & \Ss,
  \end{tikzcd} \end{center}
  where the nullhomotopy ``can'' comes from the canonical filling of the square
  \begin{center}
    \begin{tikzcd}[sep=huge]
      \sOf^{\otimes 2} \ar[r, "1 \otimes J"] \ar[d, "m"] & \sOf \ar[dl, Leftrightarrow, "\mathrm{can}"] \ar[d, "J"] \\
      \sOf \ar[r, "J"] & \Ss,
    \end{tikzcd}
  \end{center}
  which is part of the data making $J$ a map of non-unital $\mathbb{E}_\infty$-algebras.\footnote{In \cite{Boundaries} this proposition was stated with $n$ congruent to $-1$ mod $4$, however that assumption was not used in the proof.}
\end{prop}

Using this proposition, we learn through the range of interest that there are two sorts of elements in the kernel of the unit map for $\MOn$.  Specifically, elements of $\mathrm{ker}(\pi_{2n+d} \Ss \to \pi_{2n+d} \MOn)$ consist of 
\begin{enumerate}
\item Elements in the image of the map $\pi_{2n+d}J$. 
\item Elements in the image of the Toda bracket
\begin{center}
  \begin{tikzcd}[sep=huge]
    \Ss^{2n+d-1} \ar[r, "x"] \ar[rr, bend left=40, "0"{name=D}] &
    \sOf^{\otimes 2} \ar[Leftrightarrow, from=D, "\text{arbitrary}"] \ar[r, "m - 1 \otimes J"] \ar[rr, bend right=40, "0"{name=H}] &
    \sOf \ar[Leftrightarrow, from=H, "\text{can}"] \ar[r, "J"] &
    \Ss,
  \end{tikzcd}
\end{center}
where $x$ and the nullhomotopy of $x(m - 1 \otimes J)$ are arbitrary.
\footnote{Note that the indeterminacy in the given bracket is exactly elements of the first kind, so this gives an unambiguous expression for the kernel of the unit.}
We will sometimes refer to this Toda bracket as $d_2(x)$ and the condition that $x (m - 1 \otimes J) = 0$ as $d_1(x) = 0$.
\footnote{Our terminology of $d_1$ and $d_2$ comes from the spectral sequence associated to the three term complex above.}
\end{enumerate}

\begin{wrn}
The use of $J$ to denote the stable map $\Sigma^{\infty} \mathrm{O} \langle n-1 \rangle \to \Ss$ introduces an unfortunate clash of notation.  Namely, classes in the image of $\pi_*$ of this stable $J$ map need not always be in the image of the unstable $J$ homomorphism $\mathrm{O} \to \Omega^{\infty} \Ss$.
It is the image of the unstable map that it is classically referred to as $\mathrm{Im}(J)$.

We will be careful to use the phrases $\mathrm{Im}(J)$ and `the image of $J$' only in accordance with their classical meaning.
However, as in item $(1)$ above, we will need to also refer to classes in the image of the stable $J$ map.
We are comfortable with this clash of notation because we will shortly prove that, in the range of degrees of \Cref{thm:mon-unit}, there is no difference between the image of the stable $J$ map and the classical image of $J$.
We hope this does not cause the reader any undue confusion.
\end{wrn}

\section{Additional background} \label{sec:background}
Before proceeding with the proof of \Cref{thm:mon-unit}, we recall a few useful results.

\subsection{Background on non-unital algebras} \label{sec:prism}
We begin with some very general statements about non-unital $\mathbb{E}_\infty$-algebras.  Our setting will be that of an arbitrary stable and presentably symmetric monoidal category $\mathcal{C}$.  In particular, the reader is encouraged to keep the examples $\mathcal{C}=\Sp$ and $\mathcal{C}=\Syn_{\F_p}$ in mind.

\begin{ntn}
  We establish notation for some of the basic structure present on any stable presentably symmetric monoidal category $\mathcal{C}$.  There is:
  \begin{itemize}
  \item For each $k \ge 0$, an endofunctor $D_k : \mathcal{C} \to \mathcal{C}$, which is given by $X \mapsto (X^{\otimes k})_{h\Sigma_k}$,
  \item a natural transformation $c : (-)^{\otimes 2} \to D_2$,
  \item an underlying object functor $U : \mathrm{CAlg}^{\mathrm{nu}}(\mathcal{C}) \to \mathcal{C}$,
  \item a multiplication natural transformation $m : (U(-))^{\otimes 2} \to U$,
  \item a squaring natural transformation $\hat{m} : D_2(U(-)) \to U(-)$,
  \item a natural homotopy between $m$ and $c \hat{m}$.
  \end{itemize}
\end{ntn}

\newcommand{\einu}{\mathbb{E}_\infty^{\mathrm{nu}}}

Specializing all of the above to a map of non-unital $\mathbb{E}_{\infty}$-algebras $f: A \to B$, we obtain the following triangular prism:
\begin{center} \begin{tikzcd} 
    A^{\otimes 2} \ar[rr, "f \otimes f"] \ar[dd, "m"] \ar[dr, "c"] & &
    B^{\otimes 2} \ar[dd, "m" above right] \ar[dr, "c"] \\
    & D_2(A) \ar[rr, "D_2(f)" below left] \ar[dl, "\hat{m}"] & &
    D_2(B) \ar[dl, "\hat{m}"] \\
    A \ar[rr, "f"] & &
    B
\end{tikzcd} \end{center}
The relevance of this is that the homotopy ``can'' used in \Cref{sec:mon-unit} is the composite
\begin{center} \begin{tikzcd}[sep=huge]
    & A \otimes B \ar[dr, "f \otimes 1"] \\
    A^{\otimes 2} \ar[drr, phantom, "(m)"] \ar[rr, "f \otimes f", ""{name=Q}] \ar[ur, "1 \otimes f"] \ar[d, "m"] & & B^{\otimes 2} \ar[d, "m"] \\
    A \ar[rr, "f"] & & B
    \ar[phantom, from=Q, to=1-2, "(\otimes)"] 
\end{tikzcd} \end{center}
where the homotopy filling the triangle comes from the monoidal structure and the homotopy filling the square is the same one which appeared in the back face of the prism above. 
Thus, we can break ``can'' up into a composite of five different homotopies.
This factorization will lie at the heart of our ability to analyze $d_2(x)$.

\begin{rmk}[{\cite[Proposition 5.4.4.10]{HA}}] \label{rmk:nonunital-aug}
  There is an equivalence of categories between $\mathrm{CAlg}^{\mathrm{nu}}(\mathcal{C})$ and the category $\mathrm{CAlg}^{\mathrm{aug}}(\mathcal{C})$ of augmented unital $\mathbb{E}_\infty$-algebras in $\C$. On underlying objects, this equivalence sends a non-unital $\mathbb{E}_\infty$-algebra $A$ to $\mathbbm{1} \oplus A$, where $\mathbbm{1}$ is the unit.
\end{rmk}

\begin{ntn}
  Let $\coprod$ denote the coproduct in $\mathrm{CAlg}^{\mathrm{nu}} (\mathcal{C})$.
\end{ntn}

\begin{rmk} \label{rmk:nonunit-coprod}
  Since $\otimes$ is the coproduct in $\mathrm{CAlg}^{\mathrm{aug}}(\mathcal{C})$, it follows from \Cref{rmk:nonunital-aug} that $U(A \coprod B) \simeq U(A) \oplus (U(A) \otimes U(B)) \oplus U(B)$.
\end{rmk}

\subsection{Background from Goodwillie calculus}
In this subsection we specialize to the case where $\mathcal{C} = \Sp$.  Here, we can use Goodwillie calculus to gain a better understanding of the functor $U$ and the natural transformation $\hat{m}$, with the ultimate goal of understanding $\sOf$ in terms of $\on$.  For the remainder of the paper, the key takeaway is \Cref{exm:goodwillie-cof}, and the reader comfortable with that example may safely ignore the somewhat technical arguments below.
We begin by fixing some notation for the Goodwillie calculus.

\begin{ntn}
  Given a functor $F : \C \to \D$ to which we may apply the Goodwillie calculus, we denote the Goodwillie tower by
  \begin{center}
  \begin{tikzcd}
    & & \mathcal{D}_n(F) \ar[d] &
    & \mathcal{D}_2(F) \ar[d] &
    \mathcal{D}_1(F) \ar[d] \\
    F \ar[r] &
    \cdots \ar[r] &
    \mathcal{P}_n(F) \ar[r] &
    \cdots \ar[r] &
    \mathcal{P}_2(F) \ar[r] &
    \mathcal{P}_1(F) \ar[r] &
    \mathcal{P}_0(F).
  \end{tikzcd}
\end{center}
  Recall that $P_0 (F)$ may be identified with the constant functor with value $F(\ast)$, where $\ast$ is the final object of $\C$ \cite[Remark 6.1.2.2]{HA}.
\end{ntn}

Let us now analyze the underlying object functor $U$ using the Goodwillie calculus. Using the equivalence between non-unital $\mathbb{E}_\infty$-algebras and augmented $\mathbb{E}_\infty$-algebras from \Cref{rmk:nonunital-aug}, we may specialize \cite[Theorem 3.10]{KuhnTAQ} to identify the Goodwillie tower of $U$ with the following diagram.\footnote{We have displayed the tower evaluated at a object $A$. Note that as long as $A$ is 1-connective the tower is convergent.}

\begin{center}
  \begin{tikzcd} [column sep=tiny]
    & & D_n(\mathrm{TAQ}(\Ss \oplus A; \bS)) \ar[d] &
    & D_2(\mathrm{TAQ}(\Ss \oplus A; \bS)) \ar[d] &
    \mathrm{TAQ}(\Ss \oplus A; \bS) \ar[d] \\
    U(A) \ar[r] &
    \cdots \ar[r] &
    \mathcal{P}_n(U)(A) \ar[r] &
    \cdots \ar[r] &
    \mathcal{P}_2(U)(A) \ar[r] &
    \mathcal{P}_1(U)(A) \ar[r] &
    0
  \end{tikzcd}
\end{center}

\begin{ntn}
For the sake of brevity, we let $T$ denote the functor $\mathrm{TAQ}(\Ss \oplus (-); \bS)$ and $\pi: U \to T$ the natural transformation given by
\[U(A) \to P_1(U)(A) \simeq T(A).\]
\end{ntn}

\begin{rmk} \label{rmk:D2U}
  Let us explain in more detail how the identification between $\mathcal{D}_2 (U)$ and $D_2 T$ may be made.

  Let $cr_2 (U) (A,B) = \mathrm{fib} (U(A \coprod B) \to U(A) \oplus (B))$ denote the second cross-effect of $U$, and let $\Omega^\infty cr_2 (U) (\Sigma^\infty A, \Sigma^\infty B) = \varinjlim_n \Omega^{2n} cr_2 (U) (\Sigma^n A, \Sigma^n B)$ denote its linearization.
  Then the general machinery of the Goodwillie calculus identifies $\mathcal{D}_2 (U) (A)$ with $(\Omega^\infty cr_2 (U) (\Sigma^\infty A, \Sigma^\infty A))_{h \Sigma_2}$.

  Since $T$ is the linearization of $U$ by definition, to identify $\mathcal{D}_2 (U)$ with $D_2 T$ it suffices to identify the second cross-effect $cr_2 (U) (A,B)$ with $U(A) \otimes U(B)$.
  This is an immediate consequence of the identification $U(A \coprod B) \simeq U(A) \oplus (U(A) \otimes U(B)) \oplus U(B)$ of \Cref{rmk:nonunit-coprod}.
\end{rmk}

Next, we recall how calculus interacts with the natural transformation $\hat{m} : D_2 U \to U$. Taking the induced map of Goodwillie towers, we obtain a diagram:

\begin{center}\begin{tikzcd}
    & \mathcal{D}_2(D_2U) \ar[dr] \ar[dd] & &
    \mathcal{D}_1(D_2U) \ar[dr, "\simeq"] \ar[dd] \\
    D_2U \ar[dd, "\hat{m}"] \ar[r] &
    \cdots \ar[r] &
    \mathcal{P}_2(D_2U) \ar[rr] \ar[dd, "\mathcal{P}_2 (\hat{m})"] & &
    \mathcal{P}_1(D_2U) \ar[dd, "\mathcal{P}_1 (\hat{m})"] \\
    & D_2T \ar[dr] & &
    T \ar[dr, "\simeq"] \\   
    U \ar[r] &
    \cdots \ar[r] &
    \mathcal{P}_2(U) \ar[rr] & &
    \mathcal{P}_1(U).
\end{tikzcd} \end{center}

Below, we will prove a lemma that simplifies this diagram to the following:

\begin{center}\begin{tikzcd}
    & D_2T \ar[dr, "\simeq"] \ar[dd, "\simeq" below left] & &
    0 \ar[dr] \ar[dd] \\
    D_2U \ar[dd, "\hat{m}"] \ar[r] &
    \cdots \ar[r] &
    D_2T \ar[rr] \ar[dd] & &
    0 \ar[dd] \\
    & D_2T \ar[dr] & &
    T \ar[dr, "\simeq"] \\   
    U \ar[r] &
    \cdots \ar[r] &
    \mathcal{P}_2(U) \ar[rr] & &
    \mathcal{P}_1(U).
\end{tikzcd} \end{center}

\begin{lem}\label{lem:goodwillie}
  We have $\mathcal{D}_1(D_2U) \simeq 0$ and the vertical map $\mathcal{D}_2(D_2U) \to D_2T$ is an equivalence.
  Moreover, the map $D_2U \to \mathcal{P}_2(D_2U) \simeq D_2(T)$ is $D_2(\pi)$.
\end{lem}

This lemma would follow easily from an appropriate version of the chain rule in Goodwillie calculus.
Unfortunately, we could not find such a version of the chain rule in the literature, so we provide a less conceptual proof below.


\begin{proof}
  Using that $\Sp$ is stable and $D_2$ commutes with filtered colimits, we have 
  \begin{align*}
    \mathcal{D}_1(D_2U)
    &\simeq \colim_{n \to \infty} \Omega^n D_2(U( \Sigma^n - ))
    \simeq \colim_{n \to \infty} \Omega^n D_2( \Sigma^n \Omega^n U( \Sigma^n - )) \\
    &\simeq \colim_{a,b \to \infty} \Omega^a D_2( \Sigma^a \Omega^b U( \Sigma^b - ))
    \simeq \colim_{a \to \infty} \Omega^a D_2( \Sigma^a \colim_{b \to \infty} \Omega^b U( \Sigma^b - )) \\
    &\simeq \mathcal{D}_1(D_2) \circ \mathcal{D}_1(U).
  \end{align*}  
  Since $D_2$ is homogeneous of degree 2, $\mathcal{D}_1(D_2) \simeq 0$.

  In order to show that the vertical map $\mathcal{D}_2(D_2U) \xrightarrow{\mathcal{D}_2 (\hat{m})} \mathcal{D}_2 (U) \simeq D_2T$ is an equivalence it will suffice to show that the linearization of the second cross-effect functor $cr_2(\hat{m})$ is an equivalence.
  As we saw in \Cref{rmk:D2U}, there is an identification $cr_2(U)(A,B) \simeq U(A) \otimes U(B)$. 

On the other hand, after splitting $D_2(U(A))$ and $D_2(U(B))$ off of $D_2(U(A \coprod B))$ we see that $cr_2(D_2U)(A,B)$ is the direct sum of four terms:
  \begin{align*}
    cr_2(D_2U)(A,B)
    &\simeq \quad D_2(U(A) \otimes U(B)) \quad &\oplus \quad &U(A) \otimes (U(A) \otimes U(B)) \\
    &\oplus (U(A) \otimes U(B)) \otimes U(B) \quad &\oplus \quad &U(A) \otimes U(B).
  \end{align*}  

  When restricted to the final summand, $cr_2(\hat{m})$ is the identity map.
  Since the other three terms of $cr_2(D_2U)$ have connectivity which increases by 2 for every suspension (on at least one input), these terms have trivial linearizations.
  Thus, the linearization of $cr_2 (\hat{m})$ is an equivalence, as desired.

  We now need to identify the map
  \[D_2 U \to \mathcal{P}_2 (D_2 U) \simeq \mathcal{D}_2 (D_2 U) \simeq \mathcal{D}_2 (U) \simeq D_2 T\]
  with $D_2 \pi$.
  To begin, we note that since $D_2 T$ is a quadratic functor, the map $D_2 \pi$ factorizes through $D_2 U \to \mathcal{P} _2 (D_2 U) \simeq \mathcal{D}_2 (D_2 U)$.
  As a consequence, we find that it suffices to identify the maps $\mathcal{D}_2 (D_2 \pi)$ and $\mathcal{D}_2 (\hat{m})$.
  These are determined by the linearizations of the second cross-effects, so it suffices to identifies the natural transformations between these.

  Above, we determined a natural transformation $U(A) \otimes U(B) \to cr_2 (D_2 U) (A,B)$ which induces an equivalence on linearizations.
  Moreover, under the equivalence $cr_2 (U) (A,B) \simeq U(A) \otimes U(B)$ of \Cref{rmk:D2U} which identifies $\mathcal{D}_2 (U)$ with $D_2 T$, the composite
  \[U(A) \otimes U(B) \to cr_2 (D_2 U) (A,B) \xrightarrow{cr_2 (\hat{m})(A,B)} cr_2 (U) (A,B) \simeq U(A) \otimes U(B)\]
  is the identity. On the other hand, it is straightforward to see that there is a natural equivalence $cr_2 (D_2 T) (A,B) \simeq T(A) \otimes T(B)$ for which the composite
  \[U(A) \otimes U(B) \to cr_2 (D_2 U) (A,B) \xrightarrow{cr_2 (D_2 \pi)(A,B)} cr_2 (D_2 T) (A, B) \simeq T(A) \otimes T(B) \]
  is the linearization of $U(A) \otimes U(B)$.

  Combining these identifications, we find that $cr_2(\hat{m})$ and $cr_2 (D_2 T)$ agree after linearization, as desired.
%
\end{proof}

\begin{exm}\label{exm:goodwillie-cof}
Specializing to the case of $\Sigma^{\infty} \mathrm{O} \langle n-1 \rangle$ we obtain a diagram
\begin{center}\begin{tikzcd}
    D_2(\sOf) \ar[r, "\hat{m}"] \ar[d, "D_2(\pi)"] &
    \sOf \ar[d] \ar[dr, "\pi"] \\
    D_2(\on) \ar[r] &
    \mathcal{P}_2(U)(\sOf) \ar[r] &
    \on,
\end{tikzcd}\end{center}
where the vertical maps are equivalences through degree $3n-4$ for connectivity reasons and the bottom row is a cofiber sequence. The map $\pi$ can easily be identified with the mate of the identity map on $\Omega^\infty \on = \mathrm{O} \langle n-1 \rangle$. 
\end{exm}

We can take this lemma further by noting that the bottom row is split exact on homotopy groups.

\begin{lem}\label{lem:split}
  The cofiber sequence
  \[ D_2(\on) \to \mathcal{P}_2(U)(\sOf) \to \on \]
  is split exact on homotopy groups in degrees $\le 3n-4$.
\end{lem}

\begin{proof}
  It will suffice to show that the  map $\pi$ admits a section after applying $\Omega^\infty$.
  Since $\pi$ is the mate of the identity on $\Omega^\infty \on$, the unit of the $(\Sigma^{\infty},\Omega^{\infty})$-adjunction provides a section of $\Omega^{\infty} \pi$.
\end{proof}

\subsection{Remarks on the synthetic $D_2$}

For $p$ a prime, we note some useful features of the extended square functor $D_2$ on the presentably symmetric monoidal category $\Syn_{\F_p}$ of $\F_p$-synthetic spectra.

\begin{dfn}
For any integer $n$, the category of \emph{$n$-effective} $\mathbb{F}_p$-synthetic spectra is the smallest stable, full subcategory of $\Syn_{\F_p}$, closed under colimits, that contains $\nu(X)$ for every $n$-connective spectrum $X$.
\end{dfn}

This definition is useful because of the following lemma:
\begin{lem}
Suppose a synthetic spectrum $K$ is $n$-effective.  Then, the invert $\tau$ map
  \[ \pi_{t-s,t}(K) \to \pi_{t-s}(\tau^{-1} K) \]
	is an isomorphism whenever $t \leq n$.
\end{lem}

\begin{proof}
  This property is preserved under (co)fiber sequences and filtered colimits of synthetic spectra, so it suffices to check it when $K$ is $\nu(X)$ for $X$ an $n$-connective spectrum. For this, it suffices to show that
  \[\pi_{t-s,t} (\nu(X) \otimes C\tau) \cong \Ext^{s,t} _{\A} (H_* K) = 0 \]
  whenever $t < n$.
  This is an immediate consequence of the fact that $K$ is $n$-connective.
\end{proof}

Since $\tau^{-1}:\Syn_{\mathbb{F}_p} \to \Sp$ is a symmetric monoidal left adjoint, it sends the $D_2$ functor on $\Syn_{\F_p}$ to the $D_2$ functor on $\Sp$.  Furthermore, we have the following comparison between these functors:

\begin{lem} \label{lem:D2-eff}
If $X$ is an $n$-connective spectrum, and $a$ and $b$ are integers, then $D_2(\Sigma^{a,b} \nu X)$ is $(2n+2b)$-effective.  In particular, the map
  \[ \pi_{t-s,t}(D_2(\Sigma^{a,b} \nu X)) \to \pi_{t-s}(D_2(\Sigma^{a}X)) \]
  is an isomorphism for $t \leq 2n+2b$. 
\end{lem}

\begin{proof}
By definition $D_2(\Sigma^{a,b} \nu X)$ is the colimit, indexed over $BC_2$, of a diagram valued at $(\Sigma^{a,b} \nu X)^{\otimes 2}$.  Since $(2n+2b)$-effectivity is preserved under colimits, it will suffice to check that $(\Sigma^{a,b} \nu X)^{\otimes 2}$ is $(2n+2b)$-effective.  Since $\nu$ is symmetric monoidal, $(\Sigma^{a,b} \nu X)^{\otimes 2} \simeq \Sigma^{2a-2b,0} \nu(\Sigma^{b}X \otimes \Sigma^{b}X)$.  Since $X \otimes X$ is $(2n)$-connective, $\nu(\Sigma^bX \otimes \Sigma^bX)$ is $(2n+2b)$-effective, and the result follows.
\end{proof}

\begin{lem} \label{lem:D2-eff-odd}
Suppose $p \ne 2$.  Then, for any integers $a$ and $b$ and any spectrum $X$, inverting $\tau$ yields an isomorphism
\[\pi_{t-s,t}(D_2(\Sigma^{a,b} \nu X)) \to \pi_{t-s}(D_2(\Sigma^a X))\]
whenever $s \le 2b-2a$
\end{lem}

\begin{proof}
Since $p$ is odd, $D_2(\Sigma^{a,b} \nu X)$ is a summand of $(\Sigma^{a,b} \nu X)^{\otimes 2}$, in a manner compatible with the splitting of $D_2(\Sigma^a X)$ off of $(\Sigma^a X)^{\otimes 2}$.  Thus, it will suffice to check that inverting $\tau$ yields an isomorphism
\[\pi_{t-s,t}((\Sigma^{a,b} \nu X)^{\otimes 2}) \to \pi_{t-s}((\Sigma^a X)^{\otimes 2})\]
whenever $s \leq 2b-2a$

We compute, using the fact that $\nu$ is symmetric monoidal, that 
\[(\Sigma^{a,b} \nu X)^{\otimes 2} \simeq \Sigma^{2a,2b} (\nu X^{\otimes 2}) \simeq \Sigma^{0,2b-2a} \nu((\Sigma^{a} X)^{\otimes 2}).\]

For any spectrum $Y$, inverting $\tau$ gives an isomorphism $\pi_{t-s,t}(\nu Y) \to \pi_{t-s}(Y)$ when $s \le 0$, and the result follows by setting $Y=(\Sigma^{a} X)^{\otimes 2}$.
\todo{do I need to worry about completions here?}
\end{proof}

\todo{make sure I got these right}




\subsection{Background on $J$}
Our primary technique for proving that an element of the homotopy groups of spheres lies in $\mathrm{Im}(J)$ is to show that it has sufficiently high $\F_p$-Adams filtration for each prime $p$.
All of our Adams filtration bounds have essentially one source, which is a bound on the $\F_p$-Adams filtration of the stable $J$ map.
However, there is a technical complication.
Specifically, the map $J$ does not necessarily have high Adams filtration as a map from $\sOf$, but it does upon restricting to a finite skeleton of $\sOf$.
For this reason, in \cite[Section 10]{Boundaries} we systematically worked with finite skeleta.
Here, to take what we believe is a cleaner approach, we work directly in categories of truncated objects.

\begin{dfn}
  The \emph{vertical $t$-structure} on $\Syn_{\F_p}$ is the $t$-structure whose collection of connective objects is generated by $\{ \Sigma^{t-s,t}\nu X \}$ with $t-s \geq 0$ as $X$ ranges over $\Sp_{\geq 0}$. Since we may choose a cell structure on $X$ consisting of spheres of non-negative dimension, it suffices to take as generators the bigraded spheres $\Ss^{t-s,t}$ with $t-s \geq 0$.

  We write $\Syn_{\F_p}^{[a,b]}$ for the subcategory of objects concentrated in the range $[a,b]$ in the vertical $t$-structure.
\end{dfn}

The vertical $t$-structure has several pleasant properties, which we summarize in the following proposition.

\begin{prop}
  The following properties of the vertical $t$-structure hold:
  \begin{enumerate}
    \item The vertical $t$-structure is compatible with filtered colimits and tensor products.
    \item Given a synthetic spectrum $X$, the map
      \[\tau_{\geq k} X \to X\]
      is an isomorphism on $\pi_{t-s,t}$ for $t-s \geq k$.  Analogously, $\pi_{t-s, t} (\tau_{< k} X) = 0$ when $t-s \geq k$.
    \item If $\Sp$ is equipped with the usual $t$-structure, then $\nu$ is right \todo{check left v right--seems correct now} $t$-exact and $\tau^{-1}$ is $t$-exact.
    \item If $X$ is $\wt{p}$-complete\footnote{Here, $\wt{p}$ is the unique class such that $\tau \wt{p}=p$.}, then the map
      \[X \to \tau_{< k} X\]
      is an isomorphism on $\pi_{t-s,t}$ when $t-s<k$.  Analogously, $\pi_{t-s,t} (\tau_{\geq k} X) = 0$ when $t-s < k$.
  \end{enumerate}
\end{prop}

\begin{proof}
  Parts (1), (2) and (3) are immediate from the definition.
  Part (4) is a consequence of the fact that $\pi_{t-s,t} (\Ss^{0,0} _{\wt{p}}) = 0$ when $t-s < 0$.
\end{proof}

The most important consequence for us is that this means we can draw conclusions about the Adams filtration on the homotopy of a spectrum $Y$ from considering the truncation of $\nu Y$ in the vertical $t$-structure (as long as we are in the appropriate range).



We now establish the following convention, in force throughout the remainder of the paper:

\begin{cnv}
  All objects in $\Syn_{\F_p}$ will be implicitly $\tilde{p}$-completed. Similarly, all spectra will be implicitly $p$-completed, where $p$ is understood from context and tensor products will be taken in these complete categories.

  Moreover, 
  all objects will implicitly be truncated to live in $\Sp^{[0,3n-4]}$ or $\Syn_{\F_p}^{[0,3n-4]}$, and all colimits and limits are taken in these categories.
\end{cnv}

\begin{exm}
Using our new conventions, \Cref{exm:goodwillie-cof} simplifies to a cofiber sequence
  \[ D_2(\on) \simeq D_2 (\sOf) \xrightarrow{\hat{m}} \sOf \xrightarrow{\pi} \on. \]
\end{exm}

With this convention in place we can now recall the map $\wt{J}$ produced in \cite[Construction 10.5]{Boundaries}.

\begin{ntn} \label{def-M}
  Let $h(k)$ denote the number of integers $0 < s \leq k$ which are congruent to $0,1,2$ of $4$ mod $8$. We set
  \[M \coloneqq \begin{cases} h(n-1) - \lfloor \log_2 (3n-2) \rfloor + 1 & \text{ if } p = 2 \\ \mathrm{max}\left(\left\lfloor \frac{n}{2p-2} \right\rfloor - \left\lfloor \log_p \left( \frac{3n-2}{2} \right) \right\rfloor,0 \right) & \text{ if } p \neq 2 \end{cases}. \]
   Note that this notation suppresses the dependence of $M$ on $p$ and $n$.
\end{ntn}
\todo{This was stated for 4n-1, need to check it's valid as stated and used here.}

\begin{lem}[\cite{Boundaries}]
  In the truncated category of $\F_p$-synthetic spectra there exists a map
  \[ \wt{J} : \Sigma^{0,M} \nu \sOf \to \Ss \]
  which becomes $J$ upon inverting $\tau$.
\end{lem}

\begin{proof}
  \todo{I think it's OK, but would appreciate if Andy also looked at the odd primary claim here.}
	\todo{Just to be sure I understand, we are leaking a bit of Adams filtration here.  We would get slightly better bounds if we used a skeleton that depended on $d$ instead of a $(3n-3)$-skeleton, but it's a logarithmic difference.}
  Using \cite[Lemmas 10.18 and 10.19]{Boundaries}, after restricting to the $(3n-3)$-skeleton on the source the map $J$ has $\F_p$-Adams filtration at least $M$\footnote{At odd primes we have implicitly used that $2p-2$ is divisible by $4$ here.  Also, we must look at the proof of 10.19 and not just the statement.}.
	Using \cite[Lemma 9.15]{Boundaries}, after applying $\nu$ to the restricted map it becomes divisible by $\tau^M$. Now, applying our convention that everything is truncated, the difference between $\nu$ of a skeleton of $\sOf$ and $\nu \sOf$ disappears and we just obtain a map
  \[ \wt{J} : \Sigma^{0,M} \nu \sOf \to \Ss. \qedhere\]
\end{proof}




\section{Bounding Adams filtrations with synthetic lifts} \label{sec:homotopy-argument}
In this section we will show that elements in the kernel of the unit of $\MOn$ have relatively high $\F_p$-Adams filtrations for each prime $p$.

\begin{ntn}
Given an integer $0 \le d \le n-4$, we let $m$ denote $2n+d$.
Recall also our standing definition of $M$ from \Cref{def-M}.
\end{ntn}

Specifically, we will prove the following four propositions.

\begin{prop}\label{prop:d1}
  If $0 \leq d \leq n-4$ and $x \in \pi_{m}\sOf$,
  then either $J(x) \in \mathrm{Im}(J)$ or $J(x)$ is detected in $\F_2$-Adams filtration at least $2M - d-2$.
\end{prop}

\begin{prop}\label{prop:d1-odd}
  For $p \neq 2$, 
  if $0 \leq d \leq n-4$ and $x \in \pi_{m}\sOf$,
  then either $J(x) \in \mathrm{Im}(J)$ or $J(x)$ is detected in $\F_p$-Adams filtration at least $2M$.
\end{prop}

\begin{prop}\label{prop:d2}  
  Given a class $x \in \pi_{m}\left(\sOf^{\otimes 2}\right)$ where $0 \leq d \leq 2(M - 5)$ \todo{To get botany problem results, we want the case $d=-1$ here, rather than $d=0$.}
  and such that $x (m - 1 \otimes J) = 0$,
  the associated class $d_2(x) \in \pi_{m+1}(\Ss)$ is detected in $\F_2$-Adams filtration at least $2M - d - 3$.
\end{prop}

\begin{prop}\label{prop:d2-odd}
  For $p \neq 2$,
  given a class $x \in \pi_{m}\left(\sOf^{\otimes 2}\right)$ where $0 \leq d \leq n -4$
  and such that $x (m - 1 \otimes J) = 0$,
  the associated class $d_2(x) \in \pi_{m+1}(\Ss)$ is detected in $\F_p$-Adams filtration at least $2M - 2$.  
\end{prop}


\newcommand{\ud}{\underline{d}}


\begin{proof}[Proof of \Cref{prop:d1} and \Cref{prop:d1-odd}.]
	Using the splitting of $\pi_m(\sOf)$ given by \Cref{lem:split}, it will suffice to compute $J(x)$ separately for $x \in \pi_{m}(\on)$ and $x \in \pi_{m}(D_2(\on))$.

  We begin by handling the elements from $\on$.
  The composite of the splitting $\pi_{m} \on \to \pi_{m} \sOf$ with $\pi_{m}J$ is given by applying $\pi_m$ to the sequence of maps of spaces
	\[\mathrm{O}\langle n-1 \rangle \longrightarrow \Omega^{\infty} \Sigma^{\infty} \mathrm{O}\langle n-1 \rangle \stackrel{\Omega^{\infty} J}{\longrightarrow} \Omega^{\infty} \Ss,\]
	where the first map is the unit of the $(\Sigma^{\infty},\Omega^{\infty})$-adjunction.  As such, the composite is just the classical unstable $J$ homomorphism
	\[ \mathrm{O}\langle n-1 \rangle \to \mathrm{O} \to \Omega^{\infty}\Ss. \]

Now we handle the elements from $D_2(\on)$.
This means understanding what the composite
$ D_2(\sOf) \xrightarrow{\hat{m}} \sOf \xrightarrow{J} \Ss $
does on homotopy groups.
Since $J$ is a map of non-unital $\mathbb{E}_\infty$-algebras we have a commuting square
  \begin{center}
    \begin{tikzcd}
      D_2(\sOf) \ar[d, "D_2(J)"] \ar[r, "\hat{m}"] & \sOf \ar[d, "J"] \\
      D_2(\Ss) \ar[r, "\hat{m}"] & \Ss.
    \end{tikzcd}
  \end{center}
  We will bound the $\F_p$-Adams filtration of the composite $x D_2(J) \hat{m}$.
	Let $k=d+2$ if $p=2$ and $k=0$ if $p \ne 2$.
  Using the map $\wt{J}$
  we can construct the following sequence,
  \[ \Ss^{m,m+2M-k} \xrightarrow{\wt{x}} D_2(\Sigma^{0,2M}\nu \sOf) \xrightarrow{D_2(\wt{J})} D_2(\nu \Ss) \xrightarrow{\hat{m}} \nu \Ss, \]
  where $\wt{x}$ is a lift of $x$ along the isomorphism from \Cref{lem:D2-eff} (or \Cref{lem:D2-eff-odd} if $p$ is odd).
  Applying \cite[Corollary 9.21]{Boundaries} to this diagram finishes the proof.
\end{proof}

\begin{proof}[Proof (of \Cref{prop:d2} and \Cref{prop:d2-odd}).]
  In \cite[Theorem 5.2]{Boundaries}, which was recalled in \Cref{sec:mon-unit}, we identified $d_2(x)$ with $\langle x, m - 1 \otimes J, J \rangle$, where the null-homotopy on the right is the homotopy ``can'' (also discussed in \Cref{sec:mon-unit}). We will accomplish our goal by first manipulating this Toda bracket expression into a form that does not rely on the fact that $J$ is a ring map, and then lifting it to the synthetic category using $\wt{J}$. The reason we need to remove the dependence on the ring structure is that the synthetic map $\wt{J}$ is not obviously any kind of ring map. 
In order to streamline our presentation we will defer the verification of several key inputs to a sequence of lemmas after the main body of the proof.
The first of these is the following:
\begin{align}
  \text{Both } xm \text{ and } x(1 \otimes J) \text{ are nullhomotopic.}
\end{align}
Using this, we can expand $\langle x, m - 1 \otimes J, J \rangle$ into the matric form below.

\begin{center}
  \begin{tikzcd}[sep=huge]
    \Ss^{m} \ar[r, "x"] \ar[rr, bend left=40, "0"{name=A}] \ar[dr, bend right, "0"{name=B}] & \sOf^{\otimes 2} \ar[Leftrightarrow, from=A, "(1)"] \ar[Leftrightarrow, from=B, "(1)" below] \ar[r, "1 \otimes J"] \ar[d, "m"] & \sOf \ar[d, "J"] \\
    & \sOf \ar[r, "J"] \ar[ur, Leftrightarrow, "\mathrm{can}"] & \Ss
  \end{tikzcd}
\end{center}

Specializing the triangular prism produced in \Cref{sec:prism} to the map $J$ will allow us to remove dependence on ring structures from the above diagram.
Specifically, in the language of that prism we will prove the following fact, which is strictly stronger than (1):
\begin{align}
  \text{Both } xc \text{ and } x(1 \otimes J) \text{ are nullhomotopic.}
\end{align}
Assuming $(2)$, the diagram above can now be refined to the following.
Note that the size of the indeterminacy does not increase.

\begin{center}
  \begin{tikzcd}[sep=huge]
    \Ss^{m} \ar[dr, "x"] \ar[rr, "0"{name=D}] \ar[ddr, bend right, "0"{name=C} left] & & \sOf \otimes \Ss \ar[dr, "J \otimes 1"] \\
    & \sOf^{\otimes 2}  \ar[Leftrightarrow, from=C, "(2)"] \ar[Leftrightarrow, from=D, "(2)"] \ar[d, "c"] \ar[ur, "1 \otimes J"] \ar[rr, "J \otimes J"{name=E}] \ar[phantom, from=1-3, to=E, "(\otimes)"] & & \Ss^{\otimes 2} \ar[dl, "c"] \ar[dd, "\simeq"] \\
    & D_2(\sOf) \ar[urr, phantom, "(c)"] \ar[r, "D_2(J)"] & D_2(\Ss) \ar[dr, "\hat{m}"] \\
    & & & \Ss     
  \end{tikzcd}
\end{center}

Now we lift this diagram to the synthetic category. Let $k=d+2$ if $p=2$ and $k=1$ if $p \ne 2$.  The key point will be to prove the existence of synthetic lifts of the nullhomotopies from (2):
\begin{align}
  \text{Both } \tau^{k} (\Sigma^{0,2M} \nu x) c \text{ and } \tau^k (\Sigma^{0,2M} \nu x) (1 \otimes \wt{J}) \text{ are nullhomotopic.}
\end{align}
Assume that (3) holds.
Since the diagram above has no dependence on the ring structure on $J$, we can use $\wt{J}$, the nullhomotopies from (3) and the natural transformation $c$ in the category of synthetic spectra to produce the following diagram.
\begin{center}
  \begin{tikzcd}
    \Ss^{m,m+2M-k} \ar[dr, "\tau^k \Sigma^{0,2M} \nu x"] \ar[rr, "0"{name=G}] \ar[ddr, bend right, "0"{name=F}] & & \Sigma^{0,M} \nu \sOf \ar[dr, "\wt{J} \otimes 1"] \\
    & (\Sigma^{0,M} \nu \sOf)^{\otimes 2} \ar[Leftrightarrow, from=G, "(3)"] \ar[Leftrightarrow, from=F, "(3)"] \ar[ur, "1 \otimes \wt{J}"] \ar[rr, "\wt{J} \otimes \wt{J}"{name=J}] \ar[phantom, from=1-3, to=J, "(\otimes)"] \ar[d, "c"] & & (\Ss^{0,0})^{\otimes 2} \ar[dl, "c"] \ar[dd, "\simeq"] \\
    & D_2(\Sigma^{0,M} \nu \sOf) \ar[rru, phantom, "(c)"] \ar[r, "D_2(\wt{J})"] & D_2(\Ss^{0,0}) \ar[dr, "\hat{m}"] & \\
    & & & \Ss^{0,0}
  \end{tikzcd}
\end{center}

The compatibility of the natural transformation $c$ with the symmetric monoidal functor that inverts $\tau$ implies that the element of $\pi_{m+1,m+2M-k}(\Ss^{0,0})$ associated to this diagram maps to $d_2(x)$ upon inverting $\tau$.
Using \cite[Corollary 9.21]{Boundaries} to relate bigrading to Adams filtration completes the proof.
\end{proof}

In the remainder of this section, we prove the existence of the nullhomotopies (1), (2) and (3).  Their existence is immediate from the following lemma.

\begin{lem}\label{lem:nulls}
  Let $k=d+2$ if $p=2$ and $k=1$ if $p \ne 2$.
  In the situation of \Cref{prop:d2} or \Cref{prop:d2-odd},
  the following homotopy classes are trivial:
  \begin{itemize}
  \item[(a)] $xc$,
  \item[(b)] $x (1 \otimes J)$,
  \item[(c)] $\tau^{k} (\Sigma^{0,2M} \nu x) c$, and
  \item[(d)] $\tau^2 (\nu x) (1 \otimes \wt{J})$ if $p = 2$ and $\tau (\nu x) (1\otimes \wt{J})$ if $p > 2$.
  \end{itemize}
\end{lem}

\begin{proof}[Proof (reduction to (d)).]
  Clearly (d) implies (b), upon inverting $\tau$.
  By hypothesis $x(c \hat{m} - (1 \otimes J)) = 0$, and
  by \Cref{lem:split} the map $\hat{m}$ is injective on homotopy groups;
  therefore, (a) and (b) are equivalent.
  From \Cref{lem:D2-eff} (or \Cref{lem:D2-eff-odd}) we know that the map 
  \[ \pi_{m,m+2M-k}(D_2(\Sigma^{0,M} \nu \on)) \to \pi_{m}(D_2(\on)), \]
  is an isomorphism, and so (a) and (c) are equivalent.
\end{proof}

In order to prove $(d)$, the following two lemmas are helpful.

\begin{lem}\label{lem:D2-homotopy}
  At $p=2$, for $s \geq \frac{1}{2}(t-s) - n + 4$,
  every class in $\pi_{t-s,t}(\nu D_2(\on))$ which is rationally trivial is simple $\tau$-torsion.
\end{lem}

\begin{proof}
  Using the fact that $D_2(\on)$ is $2(n-1)$-connective and \cite{MayMilgram},
  we find that the bigraded homotopy groups in this region where $s \geq \frac{1}{2}(t-s)-n+3$ are determined by the $2$-Bockstein spectral sequence converging to the integral homology of
  $D_2(\on)$.
  As a consequence, to prove the lemma it suffices to show that the $2$-Bockstein spectral sequence degenerates on the $\mathrm{E}_3$-page.
  In other words, we need to prove that (the implicitly $2$-completed) $\Z \otimes D_2(\on)$ is a direct sum of shifts of $\Z$, $\Z/4$ and $\Z/2$. 

  Since $D_2$ is quadratic and $\Z \otimes \on$ is a sum of copies of $\Z$ and $\Z/2$, the equivalence
  \[ \Z \otimes D_2(\on) \simeq D_2^{\Z}( \Z \otimes \on ) \]
  allows us to conclude using the following observations:
  \begin{itemize}
  \item $\Sigma^\ell \Z \otimes_\Z \Sigma^\ell \Z$ is a shift of $\Z$.
  \item $\Sigma^\ell \Z \otimes_\Z \Sigma^\ell\F_2$ is an $\F_2$-module.
  \item $\Sigma^\ell \F_2 \otimes_\Z \Sigma^\ell \F_2$ is an $\F_2$-module.
  \item $D_2^\Z(\Sigma^{\ell} \F_2)$ consists of $4$-torsion (see \cite[Lemma A.24($\bullet$)]{torsion-bound}), i.e. consists of a sum of shifts of $\Z/4$ and $\Z/2$.
  \item $D_2^\Z(\Sigma^{\ell} \Z)$ is the cohomology of $C_2$ with $\Z$ coefficients for $\ell$ even and
    $\Z^{\mathrm{sgn}}$ coefficients for $\ell$ odd. In particular, it consists of a sum of copies of $\Z$ and $\Z/2$.
  \end{itemize}
\end{proof}

\begin{lem}\label{lem:D2-homotopy-odd}
  If $p \neq 2$, then $\pi_{t-s,t}(\nu D_2(\on))$ is $\tau$-torsion free.
  If $s > 0$, then it is $p$-torsion free as well.
\end{lem}

\begin{proof}
At odd primes $D_2(\on)$ is a summand of $\on \otimes \on$, and $\on$ is a retract of a suspension of $\mathrm{ku}$.  This lemma thus reduces to the claim that the Adams spectral sequence for $\mathrm{ku} \otimes \mathrm{ku}$ degenerates at $\mathrm{E}_2$ and is $v_0$-torsion free for $s>0$.  These claims are proved in \cite[Part III \S 17]{AdamsBlue}.
\end{proof}


\begin{proof}[Proof (of \cref{lem:nulls}(d)).]
The proof of (d) will come down to an analysis of the cofiber sequence
\[ D_2(\on) \to \sOf \to \on \]
from \Cref{exm:goodwillie-cof}.

  Let $y \coloneqq (\nu x)(1 \otimes \wt{J})$. 
  We have a diagram
  \begin{center}
    \begin{tikzcd}
      & & \Ss^{m,m+M} \ar[d, "y"] \\
      \nu D_2(\on) \ar[r] & F \ar[r] \ar[d] & \nu \sOf \ar[r] & \nu \on \\
      & E
    \end{tikzcd}
  \end{center}
  where $F$ is the fiber the of the right map and $E$ is the cofiber of the left map.
  In \cite[Lemma 11.15]{Boundaries} the authors showed that $E$ is a $C\tau$-module.
  We also know that the bigraded homotopy of $\nu \on$ is $\tau$-torsion free, since the classical Adams spectral sequence for $\mathrm{bo}$ degenerates.

  By hypothesis we have an equality $x (1 \otimes J) = x c \hat{m}$.
  Then, since $\hat{m} \pi = 0$ we learn that $x (1 \otimes J) \pi = 0$.
  Using the fact that the homotopy of $\nu \on$ is $\tau$-torsion free we can conclude that $y \nu \pi = 0$ and $y$ lifts to $F$.
  Since $E$ is a $C\tau$-module we can conclude that $\tau y$ lifts to $\pi_{m, m+M-1}\nu D_2(\on)$.
  Let $z$ be such a lift.
  
  If $p=2$, then we apply \Cref{lem:D2-homotopy} to $z$
  which applies since\footnote{This is where the upper bound on $d$ in \Cref{prop:d2} comes from.}
  \[ M-1 \geq \frac{1}{2} (2n + d) - n + 4 = \frac{d}{2} + 4 \]
  by assumption.
  From this we learn that $z$ is either simple $\tau$-torsion or
  non-$2$-torsion and non-$\tau$-torsion.
  In particular, to show that $\tau^2 y = 0$ it will suffice to show that $z$ is torsion after inverting $\tau$. Again using that the map $\hat{m}$ is injective on homotopy groups, it will suffice to show that 
  $ x (1 \otimes J)$ is torsion.
  For this we may note that $J$ becomes, upon restriction to any finite skeleton of $\sOf$, torsion as a map of spectra.
  
  If $p \neq 2$, then we use \Cref{lem:D2-homotopy-odd}, which applies when $M \geq 2$.
  If $M < 2$, then the filtration bound of \Cref{prop:d2-odd} that we are proving is vacuous.
  This time we learn that $z$ is either zero or non-$p$-torsion and non-$\tau$-torsion.
  As above using that the map $\hat{m}$ is injective on homotopy groups and
  that $J$ is torsion on any finite skeleton of $\sOf$ we may conclude that $z=0$.
\end{proof}

\section{Applications of vanishing lines} \label{sec:lines}
In this section we complete the proof of \Cref{thm:mon-unit} by showing that the Adams filtration bounds from the previous section are sufficient to conclude that $J(x)$ and $d_2(x)$ are in the image of $J$.

\begin{dfn} \label{dfn:h}
  Let $\Gamma_p(k)$ denote the minimal $m$ such that every $\alpha \in \pi_{k}\Ss_{(p)}$ with $\F_p$-Adams filtration strictly greater than $m$ is detected $K(1)$-locally.
\end{dfn}

\begin{rmk}
At odd primes, all classes in $\pi_{k} \Ss_{(p)}$ that are detected $K(1)$-locally are in $\mathrm{Im}(J)$.  At the prime $2$, a class detected $K(1)$-locally may be the sum of a class in $\mathrm{Im}(J)$ with a class in the $\mu$-family.  As explained in \cite[p. 30]{Boundaries}, no $\mu$-family class is killed by the unit of $\MOn$ when $n \geq 3$.  Specifically, composing this unit with the Atiyah--Bott--Shapiro orientation yields a sequence
  \[ \Ss \to \MOn \to \mathrm{MO}\langle 3 \rangle =\MSpin \to \mathrm{bo} \]
that on homotopy groups in degrees $\geq 2$ has the effect of killing $\mathrm{Im}(J)$ while not killing any of the $\mu$-family.  As such, any sum of an $\mathrm{Im}(J)$ class and $\mu$-family class which is in the kernel of the unit map to $\pi_*(\MOn)$ must in fact lie in $\mathrm{Im}(J)$.
\end{rmk}

As a consequence of the above remark, we will deduce \Cref{thm:mon-unit} by comparing the lower bounds of Propositions \ref{prop:d1}--\ref{prop:d2-odd} with upper bounds on $\Gamma_p$.  Such upper bounds were the main subject of \cite[Appendix B]{Boundaries}, and we recall the relevant results below.  First, we set up some notation.

\begin{ntn}\ 
\begin{enumerate}
\item $q \coloneqq 2p-2$,
\item $v_p(k)$ will denote the $p$-adic valuation of an integer $k \in \Z$,
\item 
  $ \ell(k) \coloneqq \begin{cases} v_2(k+1) + v_2(k+2) & \text{ if } p=2 \\ v_p(k+2) & \text{ if } p \neq 2 \text{ and } k+2 \equiv 0 \pmod q \\ 0 & \text{ otherwise } \end{cases}, $
\end{enumerate}
We will sometimes use that $\ell(k) \in O(\log(k))$.
\end{ntn}

Davis and Mahowald proved the bound
\begin{equation} \label{eq:DM}
  \tag{$\mathrm{DM}_2$}
  \Gamma_2(k) \leq \frac{3}{10} k + 4 + v_2(k+2) + v_2(k+1)
\end{equation}
in \cite[Corollary 1.3]{DM3}. This is the best known explicit bound at $p=2$.
In \cite[Proposition 6.3.20]{cookware}, the first author proved that
\begin{equation} \label{eq:cook3}
  \tag{$\mathrm{C}_3$}
  \Gamma_3(k) \leq \frac{1}{16}k + 7 + \frac{1}{4},
\end{equation}
\begin{equation} \label{eq:cookp}
  \tag{$\mathrm{C}_p$}
  \Gamma_p(k) \leq \frac{1}{2p^2 - 2}k + 2p^2 - 4p + 4 + \frac{2}{p^2 - 1},
\end{equation}
where the second line is for $p \geq 5$. These are explicit bounds with the best linear term that are curently known at odd primes. For our purposes we will also need an older bound with better intercept due to Gonz\'{a}lez \cite[Theorem 5.1]{Gonzalez}. For $p \geq 5$, Gonz\'{a}lez proved that
\begin{equation} \label{eq:Gon}
  \tag{$\mathrm{G}_p$}
  \Gamma_p(k) \leq \frac{(2p-1)}{(2p-2)(p^2 - p -1)}k + 3 + \ell(k).
\end{equation}

At the prime $2$, we will also make use of the bound of \cite[Appendix B]{Boundaries}, which has a better slope than the bound of Davis--Mahowald.
This bound depends on the following function which quantifies the vanishing curve of the Adams--Novikov spectral sequence.

\begin{dfn}  
  Let $f_\BP(k)$ denote the minimal $m$ such that
  for every connective $p$-local spectrum $X$, $i < k$, and $\alpha \in \pi_i(X)$,
  if $\alpha$ has $\BP$-Adams filtration at least $m$, then $\alpha = 0$.
\end{dfn}
Hopkins and Smith observed that another formulation of the Nilpotence theorem from \cite{DHS} is the fact that
$$f_{\BP}(k)=o(k).$$
As such, adding $f_{\BP}(k)$ as an ``error term'' does not affect the leading order behavior of a linear bound.

In terms of $f_{BP}$, \cite[Theorem B.7]{Boundaries} provides the following bound:
\begin{equation} \label{eq:BHS}
  \tag{$\mathrm{B}_p$}
  \Gamma_p(k) \leq \frac{(q+1)}{q|v_2|}k + \frac{(q+1)(|v_2| + 1)}{q|v_2|}f_{\BP}(k) + \ell(k).
\end{equation}


\begin{proof}[Proof (of \Cref{thm:mon-unit}(2)).]
  The possible elements in the kernel of $\pi_{m} \Ss \to \pi_m \MOn$ are of the form $J(x)$ for $x \in \pi_m \sOf$ and $d_2 (y)$ for $y \in \pi_{m-1} \sOf^{\otimes 2}$.
  Using \Cref{prop:d1} and \Cref{prop:d2}, the problem is reduced to knowing when $2M-d-2>\Gamma_2(m)$ and $d \le 2(M-5)$.
  Using the bound (\ref{eq:BHS}), the first inequality will hold when the following does:
  \[ 2 \left( h(n-1) - \lfloor \log_2(3n-2) \rfloor + 1 \right) - d -2 > \frac{1}{4}m + \frac{21}{12}f_{\BP}(m) + \ell(m)  \]
  Since $f_{\BP}$ is sublinear, $h(n)$ is $n/2$ up to a constant error term, $\ell$ is at most logarithmic, and $d \leq n$ this can be simplified to,
  \[ n - d > \frac{1}{4}(m) + \epsilon'(n) \]
  where $\epsilon'(n)$ is a sublinear error term.
  Simplifying further we get 
  \[ \frac{2}{5} n - \frac{4}{5} \epsilon'(n) > d, \]
  which is the desired conclusion.

  Similarly, $M$ is equal to $n/2$ up to a sublinear error term, so the inequality $d \leq 2(M-6)$ may be rewritten in the form
  \[ d \leq n/2 + \epsilon''(n) \]
  for a sublinear function $\epsilon'' (n)$.
  Since $1/2 > \frac{2}{5}$, we get the desired statement
\end{proof}

\begin{proof}[Proof (of \Cref{thm:mon-unit}(3)).]
  Again, using \Cref{prop:d1} and \Cref{prop:d2}, the problem is reduced to knowing when 
	$2M-d-2 > \Gamma_2(m)$
	and $d \le 2(M-5)$.
  Using the Davis--Mahowald bound (\ref{eq:DM}) on $\Gamma_2$, the first inequality holds when
  \[ 2 \left( h(n-1) - \lfloor \log_2(3n-2) \rfloor + 1 \right) - d - 2 > \frac{3}{10} (2n+d) + 4 + \log_2(2n+d+2). \]
  Using elementary manipulations one can show that it suffices to have
  \[  4n - 60 - 30\log_2(3n-2) > 13d . \]  
  On the other hand, the second inequality $d \leq 2(M-5)$ is implied by the inequality
  \[ n - 2 \log_2 (3n-2) - 10 \geq d.\]
  It is straightforward to show that the former inequality implies the latter when $n \geq 6$, so we obtain the desired result.
\end{proof}

\begin{proof}[Proof (of \Cref{thm:mon-unit}(4)).]
  Now using \Cref{prop:d1-odd} and \Cref{prop:d2-odd}, the problem is reduced to knowing when
  \[ 2M-2 > \Gamma_3(m). \]
  Using the bound (\ref{eq:cook3}) on $\Gamma_3$ above it will suffice to show that
  \[ 2\left\lfloor \frac{n}{4} \right\rfloor - 2\left\lfloor \log_3\left( \frac{3n-2}{2} \right) \right\rfloor - 2 > \frac{1}{16}(2n+d) + 7 + \frac{1}{4}. \]  
  Using that $d \leq n-4$ and rearranging it will suffice to know that
  \[ \frac{5}{16}n > 11  + 2\log_3\left( \frac{3n-2}{2} \right). \]  
  Elementary arguments now suffice to conclude that this inequality holds for $n \geq 62$.
  Finally, one may verify that the original inequality also holds for $n =60,61$ using a computer.
\end{proof}

\begin{proof}[Proof (of \Cref{thm:mon-unit}(5-9)).]
  Again using \Cref{prop:d1-odd} and \Cref{prop:d2-odd}, the problem is reduced to knowing when
  \[ 2M-2 > \Gamma_p(m). \]
  Using Gonzalez' bound (\ref{eq:Gon}) on $\Gamma_p$ it will suffice to show that
  \[   2\left\lfloor \frac{n}{2p-2} \right\rfloor - 2\left\lfloor \log_p\left( \frac{3n-2}{2} \right) \right\rfloor - 2 > \frac{(2p-1)}{(2p-2)(p^2 - p -1)}(2n+d) + 3 + \ell(2n+d).  \]
Using that $d \leq n-4$ and rearranging it will suffice to show that,
\[ \frac{2p^2 - 8p + 1}{(2p-2)(p^2 - p - 1)} n >  7 + 3\log_p(3n-2) - 2\log_p(2). \]
Elementary arguments now suffice to conclude that the inequality $2M-2 > \Gamma_p (m)$ holds in the following cases:
\begin{multicols}{2}
  \begin{itemize}
  \item $p=5$ and $n \geq 176$,
  \item $p=7$ and $n \geq 120$,
  \item $p=11$ and $n \geq 160$,
  \item $p=13$ and $n \geq 168$,
  \item $p=17$ and $n \geq 224$,
  \item $p=19$ and $n \geq 252$,
  \item $p \geq 23$ and $n \geq \frac{2p^2 - 2p -2}{3}$.
  \end{itemize}
\end{multicols}
  Since $2p^2-2p-2$ is the first degree in which the cokernel of $J$ is nontrivial we now only need to consider the primes less than $23$. Using specific knowledge of the homotopy groups of spheres in low degrees from \cite{OKA1,OKA2}, we can improve the bounds for $p \geq 11$ to the following:
\begin{multicols}{2}
  \begin{itemize}
  \item $p=11$ and $n \geq 100$,
  \item $p=13$ and $n \geq 120$,
  \item $p=17$ and all $n$,
  \item $p=19$ and all $n$.
  \end{itemize}
\end{multicols}
The bounds at $p=5,7$ could be likely be improved by making use of Ravenel's computations of the $5$-complete stable stems through 1000 \cite{GreenBook} and the Oka--Nakamura results \cite{OKA1,OKA2}. However, this is less straightforward than for larger primes.
\end{proof}

\bibliographystyle{alpha}
\bibliography{bibliography}

\end{document}